\documentclass[10pt,oneside]{amsart}
\usepackage[english]{babel}
\usepackage{amssymb}
\usepackage{amsthm}
\usepackage[foot]{amsaddr}
\usepackage{amsmath}
\usepackage{a4wide}
\usepackage{esvect}
\usepackage{hyperref}
\usepackage{verbatim}
\usepackage[dvips]{graphicx}
\usepackage{t1enc}
\usepackage{xcolor}
\usepackage[paperwidth=6.5in, paperheight=9.125in, margin=1in]{geometry}

\usepackage{tikz}
\usetikzlibrary{calc}
\tikzset{every picture/.style={line cap=none, thick} }
\usepackage{caption}
\usepackage{subcaption}
\newtheorem{thm}{Theorem}
\newtheorem{claim}[thm]{Claim}

\newtheorem{lemma}[thm]{Lemma}

\newtheorem{problem}[thm]{Problem}

\theoremstyle{definition}
\newtheorem{definition}[thm]{Definition}

\newcommand{\ignore}[1]{}

\begin{document}

\author{Ervin Gy\H ori}
\address[Ervin Gy\H ori]{Alfr\'ed R\'enyi Institute of Mathematics/Central European University}
\email[Ervin Gy\H ori]{gyori.ervin@renyi.mta.hu}
\thanks{Research  of the authors was supported by OTKA grant 116769.}

\author{Tam\'as R\'obert Mezei}
\address[Tam\'as R\'obert Mezei]{Central European University}
\email[Tam\'as R\'obert Mezei]{tamasrobert.mezei@gmail.com}

\author{G\'abor M\'esz\'aros}
\address[G\'abor M\'esz\'aros]{Alfr\'ed R\'enyi Institute of Mathematics}
\email[G\'abor M\'esz\'aros]{meszaros.gabor@renyi.mta.hu}

\title[Terminal-Pairability in Complete Graphs] 
{Terminal-Pairability in Complete Graphs}

\begin{abstract}
 \linespread{1.3}\selectfont
 We investigate terminal-pairability properties of complete graphs and improve the known bounds in two open problems. We prove that the complete graph $K_n$ on $n$ vertices is terminal-pairable if the maximum degree $\Delta$ of the corresponding demand multigraph $D$ is at most $2\lfloor\frac{n}{6}\rfloor-4$. We also verify the terminal-pairability property when the number of edges in $D$ does not exceed $2n-5$ and $\Delta\leq n-1$ holds.
\end{abstract}
\maketitle


\begin{center}
  \itshape Dedicated to the memory of our friend, professor Ralph Faudree.
\end{center}

\section{Introduction}

We discuss a graph theoretic concept of {\itshape terminal-pairability\/} emerging from a practical networking problem introduced by Csaba, Faudree, Gy\'arf\'as, Lehel, and Shelp~\cite{CS} and further studied by Faudree, Gy\'arf\'as, and Lehel~\cite{mpp,F,pp} and by Kubicka, Kubicki and Lehel~\cite{grid}. We revisit two open problems presented in~\cite{CS} and~\cite{grid}. Let $G$ be a graph
with vertex set $V(G) = T(G)\cup I(G)$ where the set $T(G)$ consists of $t$ ($t$ even) vertices of degree 1. We call $G$ a {\itshape terminal-pairable\/} network if for any pairing of the vertices of $T(G)$ there exist edge-disjoint paths in $G$ between the paired vertices. $T(G)$ is referred to as the set of {\itshape terminal nodes\/} or {\itshape terminals\/} and $I(G)$ is called the set of interior nodes of the network. Given a particular pairing of the terminals, the pairs of terminals in the pairing are simply called {\itshape pairs}. For an inner vertex $v$ we denote the number of terminal and interior vertices incident to $v$ by $d_{T(G)}(v)$ and $d_{I(G)}(v)$, respectively.

In a terminal-pairable network pairs of vertices of a graph are to be connected with edge-disjoint paths, thus the notion is clearly related to multicommodity flow problems. The concept is also related to  weakly-linked (in our case weakly-$t/2$-linked) graphs: a graph $G$ is {\itshape weakly $k$-linked\/} if, for
every pair of $k$-element sets, $X = \{x_1,\dots,x_k\}$ and
$Y = \{y_1,\dots,y_k\}$, there exist edge-disjoint
paths $P_1,\dots,P_k$, such that each $P_i$ is an $x_i y_i$-path. Observe that joining terminal vertices (leaves) to the vertices of a weakly-$k$-linked graph $G$ results in a terminal-pairable graph as long as every vertex of $G$ receives at most $k$ terminals. On the other hand, note that terminal-pairable graphs are not necessarily highly-weakly-linked. The stars (complete bipartite graph $K_{1,n}$ where one class is formed by a singleton) give a very illustrative example of terminal-pairable graphs with many terminal vertices that are not even weakly-2-linked.

Given a terminal-pairable network $G$ with a particular pairing $\mathcal{P}$ of the terminals the demand multigraph $D=(V(D),E(D))$ is defined as follows: we set $V(D)=I(G)$ and join two vertices $u,v\in V(D)$ by as many copies of the edge $uv$ as there are pairs of terminals in $\mathcal{P}$ s.t.~one vertex of the pair is joined to $u$ and the other is joined to $v$ in $G$. Obviously, $|E(D)|=\frac{|T(G)|}{2}$ and $d_D(v)= d_{T(G)}(v)$ for every $v\in V(D)$, thus in fact $\Delta(D)= \max\{ d_{T(G)}(v)\ |\ v\in I(G)\}$.
For convenience, demand multigraphs are referred to simply as demand graphs from now on.

Observe that a terminal pairing problem is fully described by the underlying network $G$ and the demand graph $D$. We call the process of substituting the demand edges by disjoint paths in $G$ the {\itshape resolution of the demand graph}.

Given a simple graph $G$, one central question in the topic of terminal-pairability is the maximum value of $t$ for which an arbitrary extension of $G$ by $t$ terminal nodes results in a terminal-pairable graphs. As at a given vertex $v\in I(G)$ at most $d_{I(G)}(v)$ edge-disjoint paths can start, the minimum degree $\delta_{I(G)}$ of the graph induced by the interior vertices provides an obvious upper bound on $t$. However, with a balanced placement of the terminals with restriction on $\Delta(D)$ of the corresponding demand graph (resembling the structure of weakly-linked graphs), the $\delta_{I(G)}$ bound on the extremal value of $t$ can be greatly improved.

Csaba, Faudree, Gy\'arf\'as, Lehel, and Shelp~\cite{CS} studied above extremal value for the complete graph $K_n$ and investigated the following question:

\begin{problem}[\cite{CS}] Let $K_n^q$ denote the graph obtained from the complete graph $K_n$ ($n$ even) by adding $q$ terminal vertices to every initial vertex. What is the highest value of $q$ (in terms of $n$) for which $K_n^q$ is terminal-pairable?
\end{problem}
One can easily verify that the parameter $q$ cannot exceed $\frac{n}{2}$. Indeed, take the demand graph $D$ obtained by replacing  every edge in a one-factor on $n$ vertices by $q$ parallel edges. In order to create edge-disjoint paths most paths need to use at least two edges in $K_n$, thus a rather short calculation implies the indicated upper bound. The so far best result on the lower bound is due to Csaba, Faudree, Gy\'arf\'as, Lehel, and Shelp:

\begin{thm}[Csaba, Faudree, Gy\'arf\'as, Lehel, Shelp~\cite{CS}]
If $q\leq \frac{n}{4+2\sqrt{3}}$, then $K_n^q$ is terminal-pairable.
\end{thm}

We improve their result by proving the following theorem:

\begin{thm}\label{delta_n/3}
If $q\leq 2\lfloor\frac{n}{6}\rfloor-4$, then $K_n^q$ is terminal-pairable.
\end{thm}

Kubicka, Kubicki and Lehel~\cite{grid} investigated terminal-pairability properties of the Cartesian product of complete graphs. In their paper the following ``Clique-Lemma'' was proved and frequently used:
\begin{lemma}[Kubicka, Kubicki, Lehel~\cite{grid}]\label{clique}
Let $G$ be a complete graph on $n$ vertices, where $n\geq 5$. If every vertex of $G$ has at most $n-1$ adjacent terminals and the total number of terminals is $2n$, then for every pairing of terminals, there are edge disjoint paths for all pairs.
\end{lemma}

In the same paper the following related problem was raised about the possible strengthening of Lemma~\ref{clique}:
\begin{problem}[\cite{grid}]
Find the largest value of $\alpha$ such that $K_n$ with $\alpha\cdot n$ terminals (at most $n-1$ at each vertex) has the above property for all $n$ larger than some constant $n_0$.
\end{problem}
Obviously, $2\leq\alpha$ due to Lemma~\ref{clique}. It is also easy to see that $\alpha < 4$. Let $D$ be a demand graph on $n\ge 4$ vertices, in which two pairs of vertices, $U,V$ and $X,Y$ are both joined by $(n-2)$ parallel edges ($2n-4$ edges or equivalently $4n-8$ terminals in total; $d_D(W)=0$ for $W\not\in\{X,Y,U,V\}$). Observe that to resolve the demand graph any disjoint path system must contain a path from $X$ to $Y$ passing through $U$ or $V$.
However, there are also $n-2$ disjoint paths connecting $U$ and $V$, meaning that $U$ or $V$ is incident to at least $2+(n-2)=n$ disjoint edges, which is clearly a contradiction. This implies that the number of terminals in $G$ cannot exceed $4n-10$. We show that this bound is sharp by proving the following theorem:

\begin{thm}\label{alpha_2n-5}
Let $D$ be a demand graph
with at most $2n-5$ edges such that no vertex is incident to more than $n-1$ edges. Then $D$ can be resolved.
\end{thm}

Before the proofs we fix further notation and terminology.
For convenience, we call a pair of edges joining the same two vertices a $C_2$. For $k>2$, $C_k$ denotes the cycle on $k$ vertices.
For a subset $S\subset V(G)$ of vertices let $d(S,V(G)-S)$ denote the number of edges with exactly one endpoint in $S$. Let $[S]$ denote the subgraph induced by the subset of vertices $S$. We call a pair of vertices joined by $k$ parallel edges a {\itshape $k$-bundle}.
For a vertex $v$ we denote the set of neighbors by $\Gamma(v)$ and use $\gamma(v)=|\Gamma(v)|$. We define  the {\itshape multiplicity\/} $m(v)$ of a vertex $v$ as follows: $m(v)=d(v)-\gamma(v)$. Observe that $m(v)$ is the minimal number of direct edges that need to be replaced by longer paths in the graph to guarantee an edge-disjoint path-system for the terminals of $v$.

\begin{figure}
	\centering
	\begin{subfigure}{.5\textwidth}
		\centering

			\begin{tikzpicture}

			\node[draw, circle] (u) at (-1,0) {$u$};
      \node[draw, circle] (v) at (1,0) {$v$};
      \node[draw, circle] (w) at (2,2) {$w$};
      \node[draw, circle] (z) at (-2,2) {$z$};

      \foreach \x/\y/\b/\c in {u/v/30/dotted,u/v/0/dashed,u/v/-30/thin,v/w/-30/thin}
			{
        \draw[\c] (\x) edge[bend left=\b] (\y);
			}
			\end{tikzpicture}

		\caption{Before}
	\end{subfigure}%
	\begin{subfigure}{.5\textwidth}
		\centering

    \begin{tikzpicture}

    \node[draw, circle] (u) at (-1,0) {$u$};
    \node[draw, circle] (v) at (1,0) {$v$};
    \node[draw, circle] (w) at (2,2) {$w$};
    \node[draw, circle] (z) at (-2,2) {$z$};

    \foreach \x/\y/\b/\c in {u/z/0/dotted,z/v/30/dotted,u/w/30/dashed,w/v/0/dashed,u/v/-30/thin,v/w/-30/thin}
    {
      \draw[\c] (\x) edge[bend left=\b] (\y);
    }
    \end{tikzpicture}

		\caption{After}
	\end{subfigure}
	\caption{Lifting 2 edges of $uv$ to $z$ and $w$}\label{fig:lifting}
\end{figure}
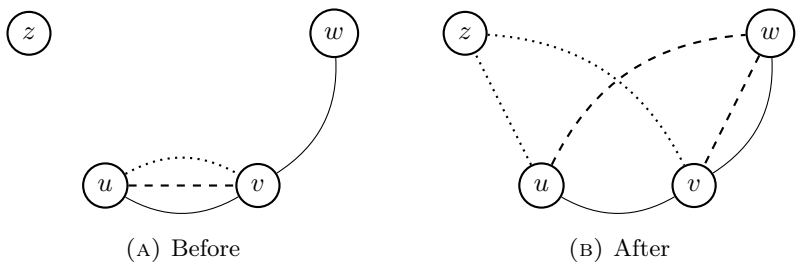

We define an operation that we will subsequently use in our proofs: given an edge ${uv}\in E$ we say that we {\itshape lift\/} ${uv}$ to a vertex $w$ when substituting the edge ${uv}$ by a path of consecutive edges ${uw}$ and ${wv}$. Note that this operation increases the degree of $w$ by 2, but does not affect the degree of any other vertex (including $u$ and $v$). Also, as a by-product of the operation, if $w$ is already joined by an edge to $u$ or $v$, the multiplicity of the appropriate pair increases by one (see Figure~\ref{fig:lifting}).

Finally, note that if a graph $G$ has $n$ vertices and $d(v)\leq n-1$, all multiplicities of $v$ can be easily resolved by subsequent liftings.
Indeed, $v$ has $n-1-\gamma(v)$ non-neighbors and $m(v)=d(v)-\gamma(v)\leq n-1-\gamma(v)$ multiplicities, thus we can assign every edge of $v$ causing a multiplicity to a non-neighbor to which that particular edge can be lifted.  We call this {\itshape resolution of the multiplicities of $v$\/} (see Figure~\ref{fig:resolve}).

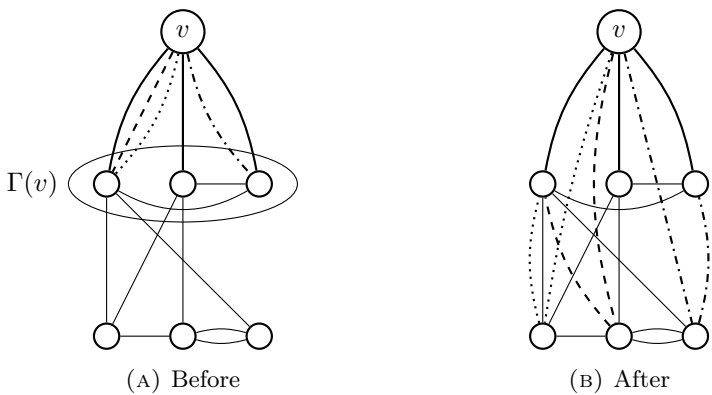
\begin{figure}
	\centering
	\begin{subfigure}{.5\textwidth}
		\centering

			\begin{tikzpicture}

			\node[draw, circle] (v) at (0,3) {$v$};

      \draw[thin] (0,1) ellipse (1.5 and 0.5);

      \foreach \x in {0,1,2}
        \node[draw, circle] (u\x) at (\x-1,1) {};

      \foreach \x in {0,1,2}
        \node[draw, circle] (w\x) at (\x-1,-1) {};

      \foreach \x/\b/\c in {0/15/dotted,0/0/dashed,0/-15/thick,1/0/thick,2/-15/dashdotted,2/15/thick}
      {
        \draw[\c] (v) edge[bend left=\b] (u\x);
      }

      \foreach \x/\y/\b in {u0/u2/-30,u1/u2/0,w1/w2/15,w1/w2/-15,u0/w0/0,w0/w1/0,u0/w2/0,w0/u1/0,u1/w1/0}
        \draw[thin] (\x) edge[bend left=\b] (\y);

      \node[anchor=east] (label_u) at (-1.5,1) {$\Gamma(v)$};
      \node[anchor=west,white] (label_u_invisible) at (1.5,1) {$\Gamma(v)$};

      \end{tikzpicture}

		\caption{Before}
	\end{subfigure}%
	\begin{subfigure}{.5\textwidth}
		\centering

    \begin{tikzpicture}

    \node[draw, circle] (v) at (0,3) {$v$};

    \foreach \x in {0,1,2}
      \node[draw, circle] (u\x) at (\x-1,1) {};

    \foreach \x in {0,1,2}
      \node[draw, circle] (w\x) at (\x-1,-1) {};

    \foreach \x/\b/\c in {0/-15/thick,1/0/thick,2/15/thick}
    {
      \draw[\c] (v) edge[bend left=\b] (u\x);
    }

    \foreach \x/\y/\b in {u0/u2/-30,u1/u2/0,w1/w2/15,w1/w2/-15,u0/w0/0,w0/w1/0,u0/w2/0,w0/u1/0,u1/w1/0}
      \draw[thin] (\x) edge[bend left=\b] (\y);

    \foreach \x/\y/\b/\c in {v/w0/-5/dotted,w0/u0/15/dotted,v/w1/-15/dashed,w1/u0/15/dashed,v/w2/0/dashdotted,w2/u2/-15/dashdotted}
    {
      \draw[\c] (\x) edge[bend left=\b] (\y);
    }

    \end{tikzpicture}
		\caption{After}
	\end{subfigure}
	\caption{Resolving the multiplicities at $v$}\label{fig:resolve}
\end{figure}

\section{Proof of Theorem~\ref{delta_n/3}}

We show that if $D=(V,E)$ is a demand multigraph on $n$ vertices and $\Delta(G)\leq 2\lfloor\frac{n}{6}\rfloor-4$, then $D$ can be transformed into a simple graph by replacing parallel edges by paths of $D$. We prove the statement by induction on $n$. Observe first that the statement is obvious for $n < 18$. For $18\leq n < 24$, note that the demand graph $D$ is the disjoint union of 2-bundles, circles, paths, and isolated vertices. It is easy to see that multiplicities in these demand graphs can be resolved; we leave the verification of the statement to the reader.

From now on assume $n\geq 24$. We may assume without loss of generality that $D$ is an $\big(2\lfloor\frac{n}{6}\rfloor-4\big)$-regular multigraph; if necessary, additional parallel edges may be added to $D$. Should a single vertex $v$ fail to meet the degree requirement, we bump up its degree by further lifting operations as follows: as the deficit $\big(2\lfloor\frac{n}{6}\rfloor-4\big)-d(v)$ must be even, we can lift an arbitrary edge $e\in E([V(D)-v])$ to $v$. We remind the reader that lifting $e$ to $v$ increases $d(v)$ by two while it does not affect the degree of the rest of the vertices.

We will use the well known 2-Factor-Theorem of Petersen~\cite{Petersen}. Be aware that a 2-factor of a multigraph may contain several $C_2$'s (however, this is the only way parallel edges may appear in it).
\begin{thm}[\cite{Petersen}]\label{Petersen}
Let $G$ be a $2k$-regular multigraph. Then $E(G)$ can be decomposed into the union of $k$ edge-disjoint $2$-factors.
\end{thm}
Some operations, which are performed later in the proof, are featured in the following definition, claim, and lemma.
\begin{definition}[Lifting coloring]
Let $F$ be a multigraph, and $c:E(F)\cup V(F)\to \{1,2,3\}$ be a coloring of the edges and vertices of $F$. We call $c$ a {\itshape lifting coloring\/} of $F$ if and only if
\begin{enumerate}
\item for any edge $e=uv\in E(F)$, $c(u)\neq c(e)$ and $c(v)\neq c(e)$, and
\item for any two edges $e_1,e_2\in E(F)$ incident to a common vertex we have $c(e_1)\neq c(e_2)$.
\end{enumerate}
Moreover, if the number of vertices in different color classes differ by either 0, 1, or 2, then we call $c$ a {\itshape balanced lifting coloring\/} of $F$.
\end{definition}

\begin{claim}\label{claim:liftcolor}
  Let $F$ be a multigraph such that $\forall v\in V(F)$ we have $d_F(v)\le 2$. If $w_1,w_2,w_3\in V(F)$ are three pairwise non-adjacent different vertices, then $F$ has a balanced lifting coloring where $w_i$ gets color $i$.
\end{claim}
\begin{proof}
The proof is easy but its complete presentation requires a rather lengthy (but straightforward) casework. We leave the verification of the statement to the reader. Figure~\ref{fig:liftcoloring} shows an example output of this lemma.
\begin{figure}
		\centering

			\begin{tikzpicture}
      \def \uno {thick}
      \def \due {dotted}
      \def \tre {dashed}
      \def \dotsize {5pt}

      \foreach \x/\c in {1/\uno,2/\due,3/\tre}
        \node[draw, circle, \c] (x\x) at (\x-2,3) {$x_\x$};


      \def \2fshift {1};

      \foreach \i/\c in {1/\uno,2/\due,3/\due}
        \node[shift={(-3,1)}, circle, draw, inner sep=\dotsize, \c] (b\i) at ({+90+90* \i}:1) {};

      \foreach \i/\c in {1/\due,2/\due,3/\uno,4/\tre}
        \node[shift={(0,0.5)}, circle, draw, inner sep=\dotsize, \c] (c\i) at ({+90+72* \i}:1) {};

      \foreach \i/\c in {1/\tre,2/\uno,3/\uno}
        \node[shift={(2,0.5)},draw, circle, inner sep=\dotsize, \c] (d\i) at (1,2.5-\i) {};

      \foreach \i/\j/\c in {b1/x1/\due,x1/b3/\tre,b3/b2/\uno,b2/b1/\tre}
        \draw (\i) edge[very thick,\c] (\j);

      \foreach \i/\j/\c in {c4/x2/\uno,x2/c1/\tre,c1/c2/\uno,c2/c3/\tre,c3/c4/\due}
        \draw (\i) edge[very thick,\c] (\j);

      \foreach \i/\j/\c in {d1/x3/\due,x3/d1/\uno,d2/d3/\due,d3/d2/\tre}
        \draw (\i) edge[very thick, bend left=15,\c] (\j);

      \end{tikzpicture}

		\caption{A balanced lifting coloring, where $x_1,x_2,x_3$ get pairwise different colors.}\label{fig:liftcoloring}
\end{figure}
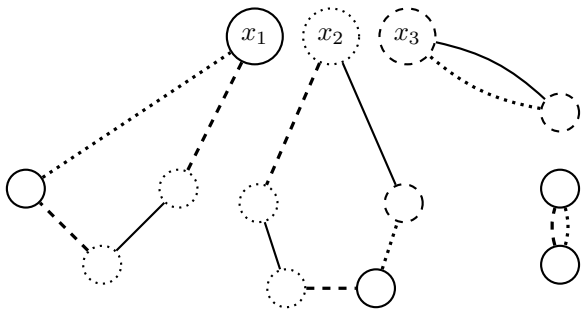
\end{proof}

\begin{lemma}\label{lemma:main}
  Let $D$ be a demand graph on $n$ vertices, such that $\Delta(D)\le \lfloor\frac{n}{3}\rfloor-4$. Furthermore, let $X=\{x_1,x_2,x_3\}$ be a subset of $V(D)$ of cardinality 3, such that $|E(D[X])|=0$. Let $B$ be an at most 3 element subset of $V(D)\setminus X$. Let $F$ be an $\le 2$-factor of $D$. Then there exists a demand graph $H$ which satisfies
    \begin{itemize}
      \item $V(H)=V(D)\setminus X$,
      \item $E(H)\supset E(D[V(H)])\setminus F$,
      \item $\{e\in E(H):\ e\text{ is incident to at least one of }B\}\subset E(D)$, and
      \item for any $v\in V(H)$ we have $d_H(v)\le d_D(v)-d_F(v)(+1\text{ if } v\notin B)$.
    \end{itemize}
    Moreover, if $H$ has a resolution, then so does $D$.
\end{lemma}
\begin{proof}
  We will perform a series of liftings in $D$ in two phases, obtaining $D'$ and $D''$. At the end of the second phase, we will achieve that $X$ has no parallel edges in $D''$. Therefore setting $H=D''-X$ will satisfy the second claim of the lemma.

  First, we determine the series of liftings to be executed in the first phase. Notice that Claim~\ref{claim:liftcolor} implies the existence of a balanced lifting coloring $c$ of $F$ such that $c(x_i)\equiv i+1 \pmod 3$. Lift each edge $f\in F$ to $x_{c(f)}$, except if $f$ is incident to $x_{c(f)}$. Let $F'$ be the set of lifted edges, that is
  $$F'=\bigcup_{\substack{f\in F,\\ x_{c(f)}\notin f}}\Big\{\text{two edges joining $x_{c(f)}$ to the two vertices of $f$}\Big\},$$
  where $\dot\cup$ denotes the disjoint union. Let the multigraph $D'$ be defined on the same vertex set as $D$, and let its edge set be $$E(D')=\{e\in E(D):\ e\notin F\text{ or }x_{c(e)}\notin e\}\dot\cup F'.$$ In other words, $D'$ is the demand graph into which $D$ is transformed by lifting the elements of $F$. Let $Y=V(D)\setminus X$. Observe that $d_{D'}(y)=d_{D}(y)$ for $y\in Y$.
  Let $Y_i=\{y\in Y\setminus B\ |\ c(y)=i\}$ be the color $i$ vertices in $Y\setminus B$.  The balancedness of $c$ guarantees that $$|Y_i|=|c^{-1}(i)\setminus\{x_{i-1}\}\setminus B|\ge |c^{-1}(i)|-1-|B|\ge \left\lfloor\frac{n}{3}\right\rfloor -5.$$

  In the second phase, our task is to resolve all multiplicities of $x_i$ in $D'$. Observe that as edges of $F$ of the same color formed a matching, out of every two parallel edges that are incident to $x_i$ in $D'$ at least one of them must be an initial edge in $E(D')\setminus F'$. The vertex $x_i$ is incident to $d_{D'}(x_i)-d_{F'}(x_i)$ edges of $E(D')\setminus F'$; we plan to lift these edges to the elements of $Y_i$ by using every vertex in $Y_i$ for lifting at most once. Since $$d_{D'}(x_i)-d_{F'}(x_i)\le d_{D}(x_i)-1=\Delta(D)-1\le \left\lfloor\frac{n}{3}\right\rfloor-5\le |Y_i|,$$
  and elements of $Y_i$ are not incident to edges of color $i$, the set $Y_i$ offers enough space to carry out the liftings. That being said, note that neighbors of $x_i$ in $Y_i$ cannot be used for lifting as they would create additional multiplicities. On the other hand, if $v\in Y_i$ and  $e={vx_i}\in E(D)$ then $e$ is an initial edge of $x_i$ that either generates no multiplicity at all or it is part of a bundle of parallel edges, one of which we do not lift. In other words, for every vertex of $Y_i$ that is excluded from the lifting we mark an initial edge of $x_i$ that we do not need to lift. As a result of this, resolution of the remaining multiplicities at $x_i$ can be performed in $Y_i-\Gamma(x_i)$. Let $D''$ denote the demand graph obtained after resolving all of the multiplicities of $x_1$, $x_2$, and $x_3$.

  At most 1 element of $E(D')\setminus E(F')$ has been lifted to each $y\in Y$, therefore there are no multiple edges between the sets $X$ and $Y$ in the demand graph $D''$. Moreover, $D''[X]=D'[X]$ is a subgraph of a triangle, which emerges as we lift the at most one edge of color $i+2$ of $x_i$ to $x_{i+1}$ (take the indices cyclically), for $i=1,2,3$.

  Any vertex $y\in Y$ of color $i$ has at most two incident edges in $F'$, joining $y$ to a subset of $\{x_{i+1},x_{i+2}\}$.
  \begin{itemize}
    \item If an edge has been lifted to $y\in Y$ of color $i$, then $y$ is adjacent to $x_i$ and $d_{D''}(y)=d_{D'}(y)+2$. Thus $y$ is joined to at least $d_{F'}(y)+1$ elements of $X$ in $D''$. As no edge of color $i$ can be incident to $y$, we have $d_{F'}(y)=d_F(y)$. Therefore
    $$d_{D''[Y]}(y)\le d_{D''}(y)-d_{F'}(y)-1=d_{D'}(y)-d_{F}(y)+1=d_{D}(y)-d_{F}(y)+1.$$
    \item If no edges have been lifted to $y\in Y$, then $d_{D''}(y)=d_{D'}(y)$ and $y$ is adjacent to at least $d_{F}(y)$ elements of $X$ in $D''$. Therefore
    $$d_{D''[Y]}(y)=d_{D''}(y)-d_F(y)\le d_{D'}(y)-d_F(y)=d_{D}(y)-d_F(y).$$
  \end{itemize}
  As elements of $B$ are excluded from $Y_i$, 0 edges are lifted to them, and so we proved the statement of the lemma.
\end{proof}

Let $X_1=\{x_1,x_2,x_3\}$ be a subset of 3 elements of $V(D)$, such that $D[X]$ has 0 edges. Such a set trivially exists, as any two non-adjacent vertices have $(n-2)-2\Delta(D)\ge \frac{n}{3}+2$ common non-neighbors. Since the degree in $D$ is at least $2\cdot(24/6)-4=4$, Theorem~\ref{Petersen} implies the existence of two disjoint 2-factors, $A_1$ and $A_2$ of $D$.
Notice that $A_2-X$ has 3 path components (as a special case, an isolated vertex is a path on one vertex).
Extend $A_2-X$ to a maximal $\le 2$-factor $F_2$ of $D-A_1$. It is easy to see that there exists a 3-element subset $B_1$ of $V(D)\setminus X$ such that \begin{itemize}
  \item $B_1$ induces 0 edges in $D-A_1$,
  \item $\{v\in V(D)\setminus X:\ d_{F_2}(v)= 0\}\subset B_1$, and
  \item $B_2=\{v\in V(D)\setminus X:\ d_{F_2}(v)=1\}\setminus B_1$ has cardinality at most 3.
\end{itemize}

We are ready to use Lemma~\ref{lemma:main}. First, apply it to $D$, where we lift $F=A_1$ to elements of $X=X_1$, while not creating new edges incident to $B=B_1$. Let the obtained graph be $H_1$. We have $\Delta(H_1)\le \Delta(D)-\delta(A_1)+1=\Delta(D)-1$. Furthermore, $E(H_1[B_1])\subseteq E(D[B_1])=\emptyset$.

We apply Lemma~\ref{lemma:main} once more. Now $H_1$ is our base demand graph, $F_2$ is the $\le 2$-factor to be lifted to elements of $B_1$, and we avoid lifting to elements of $B_2$. Let the resulting demand graph be $H_2$, whose vertex set is $V(D)\setminus X\setminus B_1$ of cardinality $n-6$. We have
\begin{align*}
  d_{H_2}(v)&\le\left\{
  \begin{array}{ll}
      d_{H_1}(v)-d_{F_2}(v)+1 & \text{ if }v\notin B_2, \\
      d_{H_1}(v)-d_{F_2}(v) & \text{ if }v\in B_2.
    \end{array}
    \right.\le \\
  &\le\left\{\begin{array}{ll}
      (\Delta(D)-1)-2+1 & \text{ if }v\notin B_2, \\
      (\Delta(D)-1)-1 & \text{ if }v\in B_2.
  \end{array}
  \right.\le \\
  &\le\Delta(D)-2=2\left\lfloor\frac{n-6}{6}\right\rfloor-4.
\end{align*}

By induction on $n$, we know that $H_2$ has a resolution, implying that $H_1$ has a resolution, which in turn implies that $D$ has a resolution.

\section{Proof of Theorem~\ref{alpha_2n-5}}
We prove our statement by induction on $n$. For $n\leq 4$ the statement is straightforward, the cases $n=5,6$ require a somewhat cumbersome casework. Note that if $n\geq 4$ we may assume $D$ has exactly $2n-5$ edges, otherwise we join two non-neighbors whose degree is smaller than $n-1$.

For the inductive step, we choose a vertex $x$, resolve all of its multiplicities, and delete it from the demand graph. There are two additional conditions to assert as the number of vertices decreases from $n$ to $n-1$:
\begin{itemize}
\item[i)] We need to delete at least 2 edges from $D$. These edges can be either already incident to $x$ or can be lifted to $x$.
\item[ii)] Let $B$ denote the set of vertices of degree greater than or equal to $n-1$. Obviously, to apply induction we need to decrease the degree $d(v)$ of every vertex $v\in B$ by at least one. Decreasing $d(v)$ can be performed by lifting an edge incident to $v$ to $x$. Note that this operation might create additional multiplicities that need to be resolved before the deletion of $x$.
\end{itemize}
In addition, observe that we can lift at least one edge to a vertex $v$ without its degree exceeding the degree bound for $n'=n-1$ if and only if $d(v)< n-2$. Let
\[B=\{z_1,\dots,z_{|B|}\}=\{v\in V(D): d(v)\geq n-2\}.\]
As $\sum\limits_{v\in V(D)}d(v) = 4n-10$, it follows that $|B|\leq 3$. We perform a casework on $|B|$.

\begin{itemize}
\item[$|B|=0$:] If $B$ is empty, then the only condition we need to guarantee is the deletion of at least two edges in $D$.  We have two cases.
\begin{itemize}
  \item \textit{If there is an $x\in V(D)$ with $\gamma(x)\ge 2$:} we have $n-1-\gamma(x)$ vertices for lifting to resolve the $d(x)-\gamma(x)$ multiplicities of $x$. Obviously, $d(x)-\gamma(x)\leq n-3 -\gamma(x)$ thus we have enough space to resolve all multiplicities of $x$. After the deletion of $x$, the graph has lost $\gamma(x)\ge 2$ edges, and the maximum degree is still two less then the number of vertices.
  \item \textit{If $\forall x\in V(D)$ we have $\gamma(x)\le 1$,} then $D$ is the disjoint union of bundles and isolated vertices, which is trivial to resolve.
\end{itemize}

\item[$|B|=1$:] We perform the same operation as in the previous case with the choice $x=z_1$. Observe that  our inequality becomes $d(z_1)-\gamma(z_1)\leq n - 1 -\gamma(z_1)$ thus we have enough vertices in the multigraph to perform all the necessary liftings.

\item[$|B|=2$:] Observe first that $z_1$ and $z_2$ are joined by an edge $e$ or else \[2n-5\geq d(B,V(D)-B)= d(z_1)+d(z_2)\geq 2n-4,\] a contradiction. Let us first assume that $z_1$ or $z_2$ (say, $z_1$) has an edge ending in a vertex different from $z_2$ (i.e. $d(B,V(D)-B)>0$). Observe that in this case $m(z_1)=d(z_1)-\gamma(z_1)\leq (n-1)-\gamma(z_1)$, thus all multiplicities of $z_1$ can be resolved by lifting the appropriate edges to $V(D)-\{z_1\}-\Gamma(z_1)$.

In the remaining case $z_1$ and $z_2$ form a bundle of at most $n-1$ edges. We can lift $n-2$ of these edges to $V(D)-B$ without difficulties, delete one of the vertices in $B$, and proceed by induction.

\item[$|B|=3$:] Observe that any two vertices of $\{z_1,z_2,z_3\}$ must be joined by and edge else the same reasoning as above leads to contradiction. Note also that a simple average degree calculation guarantees the existence of an isolated vertex $x$. We distinguish two cases:
\begin{itemize}
\item[i)] If $d(B,V(D)-B)=0$, we may assume that $V(D)-B$ contains an edge, otherwise $3(n-3)\geq 4n-10\Rightarrow n\leq 7$ and all edges are contained in $B$. For $n=5,6,7$ that leads to 4 possible demand graphs whose resolution can be easily completed; a case for $n=6$ is shown in Figure~\ref{fig1}.

\begin{figure}
	\centering
	\begin{subfigure}{.5\textwidth}
		\centering

			\begin{tikzpicture}
			\def \radius {1.5cm}

			\foreach \s in {1,2,3}
			{
				\node[draw, circle] (\s) at ({90+120 * \s}:\radius/2) {};
				\node[draw, circle] (O\s) at ({30+360/3 * \s}:1.4*\radius) {};
			}
			\foreach \x/\y/\c in {1/2/thin,2/1/dotted,1/3/dashed,3/1/thin,2/3/thin,3/2/dashdotted}
			\draw[\c] (\x) edge[bend left=15] (\y);

			\draw[thick] (1) edge (2);
      \draw[color=white] (O1) edge[bend left=15] (O3);
			\end{tikzpicture}

		\caption{Demand graph}\label{fig1:before}
	\end{subfigure}%
	\begin{subfigure}{.5\textwidth}
		\centering

			\begin{tikzpicture}

			\def \radius {1.5cm}

			\foreach \s in {1,2,3}
			{
				\node[draw, circle] (\s) at ({90+120 * \s}:\radius/2) {};
				\node[draw, circle] (O\s) at ({30+360/3 * \s}:1.4*\radius) {};
			}

   			\foreach \x/\y in {1/2,2/3,3/1}
				\draw[thin] (\x) edge[bend left=15] (\y);

			\foreach \x/\y/\z/\c in {1/2/O2/dotted,2/3/O3/dashdotted,3/1/O1/dashed}
			{
				\draw[\c] (\x) edge (\z);
				\draw[\c] (\y) edge (\z);
			}

			\draw[thick] (1) edge (O3);
			\draw[thick] (2) edge (O1);
			\draw[thick] (O1) edge[bend left=15] (O3);

			\end{tikzpicture}

		\caption{A solution}\label{fig1:after}
	\end{subfigure}
	\caption{}\label{fig1}
\end{figure}

Let $f$ denote an arbitrary edge in $V(D)-B$. We lift two edges of $B$ not belonging to the same pair as well as $f$ to $x$; observe that the degrees of all vertices in $B$ dropped by at least 1. As $n\geq 7$, the multiple edge created at vertex $x$ can be lifted to a vertex of $V(D)-B$ that was not incident to $f$.

\item[ii)] If $d(B,V(D)-B) > 0$ let $f$ be an edge between $B$ and $V(D)-B$. Without loss of generality we may assume $f$ is incident to $z_3$. We lift $f$ as well as an edge $e$ between $z_1$ and $z_2$; as $e$ and $f$ are disjoint, no new multiplicity is created, thus we can proceed by induction.
\end{itemize}
\end{itemize}

\section*{Acknowledgment}
We would like to heartfully thank Professor Jen\H o Lehel for drawing our attention to the above discussed problems. The third author in particular wishes to express his gratitude for the intriguing discussions about terminal-pairability problems.

The first and third authors of this current paper consider themselves especially fortunate to have worked with and have been inspired by the work of Professor Ralph Faudree. We would like to dedicate our paper to the memory of Professor Faudree as a humble contribution to one of his favorite research topics.

\bibliographystyle{acm}
\bibliography{refs}

\end{document}